\documentclass{amsart}

\usepackage{amsmath,amssymb,amsfonts,enumerate,amsthm,graphicx,color}

\def\opn#1#2{\def#1{\operatorname{#2}}} 
\opn\chara{char} \opn\length{\ell} \opn\pd{pd} \opn\rk{rk}
\opn\projdim{proj\,dim} \opn\injdim{inj\,dim} \opn\rank{rank}
\opn\depth{depth} \opn\grade{grade} \opn\height{height}
\opn\embdim{emb\,dim} \opn\codim{codim}

\opn\Tr{Tr} \opn\bigrank{big\,rank}
\opn\superheight{superheight}\opn\lcm{lcm}
\opn\trdeg{tr\,deg}%
\opn\reg{reg} \opn\lreg{lreg} \opn\skel{skel}
\opn\multideg{multideg}
\opn\div{div} \opn\Div{Div} \opn\cl{cl} \opn\Cl{Cl}
\opn\Spec{Spec} \opn\Supp{Supp} \opn\supp{supp} \opn\Sing{Sing}
\opn\Ass{Ass}
\opn\Ann{Ann} \opn\Rad{Rad} \opn\Soc{Soc}
\opn\Ker{Ker} \opn\Coker{Coker} \opn\Im{Im} \opn\Hom{Hom}
\opn\Tor{Tor} \opn\Ext{Ext} \opn\End{End} \opn\Aut{Aut}
\opn\id{id}

\opn\nat{nat}
\opn\pff{pf}
\opn\Pf{Pf} \opn\GL{GL} \opn\SL{SL} \opn\mod{mod} \opn\ord{ord}
\opn\aff{aff} \opn\con{conv} \opn\relint{relint} \opn\st{st}
\opn\lk{lk} \opn\cn{cn} \opn\core{core} \opn\vol{vol}
\opn\link{link} \opn\star{star} \opn\skel{skel} \opn\Reg{Reg}
\opn\gr{gr}

\def\pot#1#2{#1[\kern-0.28ex[#2]\kern-0.28ex]}

\opn\dirlim{\underrightarrow{\lim}}
\opn\inivlim{\underleftarrow{\lim}}
\def\Implies{\ifmmode\Longrightarrow \else
     \unskip${}\Longrightarrow{}$\ignorespaces\fi}
\def\implies{\ifmmode\Rightarrow \else
     \unskip${}\Rightarrow{}$\ignorespaces\fi}
\def\iff{\ifmmode\Longleftrightarrow \else
     \unskip${}\Longleftrightarrow{}$\ignorespaces\fi}

\let\:=\colon

\newtheorem{thm}{Theorem}[section]
\newtheorem{cor}[thm]{Corollary}
\newtheorem{lem}[thm]{Lemma}
\newtheorem{prop}[thm]{Proposition}

\newtheorem{rem}[thm]{Remark}

\numberwithin{equation}{section}

\begin{document}
\bibliographystyle{amsplain}

\title{ The edge ideals of complete multipartite hypergraphs }
\author{ Dariush Kiani and Sara Saeedi Madani }
\thanks{2010 \textit{Mathematics Subject Classification.} 05E45, 13F55, 05C65, 13H10.}
\thanks{\textit{Key words and phrases.} Complete $s$-uniform $t$-partite hypergraphs, Edge ideals, (Sequentially) Cohen-Macaulay,
Vertex decomposable. }

\address{Dariush Kiani, Department of Pure Mathematics,
 Faculty of Mathematics and Computer Science,
 Amirkabir University of Technology (Tehran Polytechnic),
424, Hafez Ave., Tehran 15914, Iran, and School of Mathematics, Institute for Research in Fundamental Sciences (IPM),
P.O. Box 19395-5746, Tehran, Iran.} \email{dkiani@aut.ac.ir, dkiani7@gmail.com}
\address{Sara Saeedi Madani, Department of Pure Mathematics,
 Faculty of Mathematics and Computer Science, Amirkabir University of Technology (Tehran Polytechnic),
424, Hafez Ave., Tehran 15914, Iran, and School of Mathematics, Institute for Research in Fundamental Sciences (IPM),
P.O. Box 19395-5746, Tehran, Iran.} \email{sarasaeedi@aut.ac.ir}

\begin{abstract}
We classify all complete uniform multipartite hypergraphs with respect to some algebraic properties, such as being (almost) complete intersection, Gorenstein, level, $l$-Cohen-Macaulay, $l$-Buchsbaum, unmixed, and satisfying Serre's condition $S_r$, via some combinatorial terms. Also, we prove that for a complete $s$-uniform $t$-partite hypergraph $\mathcal{H}$, vertex decomposability, shellability, sequentially $S_r$ and sequentially Cohen-Macaulay properties coincide with the condition that $\mathcal{H}$ has $t-1$ sides consisting of a single vertex. Moreover, we show that the latter condition occurs if and only if it is a chordal hypergraph.
\end{abstract}

\maketitle

\section{ Introduction } \label{introduction}

\noindent In recent years, many authors have focused on studying different kinds of monomial ideals associated to combinatorial objects, such as Stanley-Reisner ideals, facet ideals, edge ideals and path ideals (see for example \cite{HTYZ},\cite{MKM}, \cite{MK}, \cite{SKT}, \cite{W1} and \cite{W}). In this paper, we study edge ideals of  hypergraphs.

A \textbf{hypergraph} $\mathcal{H}$ with finite vertex set $V(\mathcal{H})$ is a family of nonempty subsets of $V(\mathcal{H})$ whose union is $V(\mathcal{H})$, called edges. The set of vertices and edges of $\mathcal{H}$ are denoted
by $V(\mathcal{H})$ and $E(\mathcal{H})$, respectively. Sometimes, we also use $\mathcal{H}$ as its set of edges. An \textbf{induced subhypergraph} of $\mathcal{H}$, over $S\subseteq V(\mathcal{H})$, is defined as ${\mathcal{H}}_S=\{e\in \mathcal{H}~:~e\subseteq S\}$. If all the edges of a hypergraph $\mathcal{H}$ have the same cardinality $t$, then it is said that $\mathcal{H}$ is $t$-\textbf{uniform} (also sometimes referred as a
$t$-graph). We call an edge of cardinality one, an \textbf{isolated} vertex. If none of the edges of $\mathcal{H}$ is included in another, then $\mathcal{H}$ is called a \textbf{simple} hypergraph. Throughout this paper, we mean by a hypergraph, a simple one.

In this paper, we focus on a class of uniform hypergraphs called $t$-partite which is a generalization of multipartite graphs.
These hypergraphs are important from the combinatorial point of view. The organization of this paper is as follows. In Section~\ref{Preliminaries}, we bring some definitions and general properties of hypergraphs and simplicial complexes, and also some relations between combinatorial objects and algebraic ones. For this purpose, we mostly use \cite{Be} and \cite{HH}. Moreover, we
introduce the class of multipartite hypergraphs, and in the case of complete multipartite hypergraphs, we pose some of their basic properties, which will be used in the other sections. In Section~\ref{Algebraic}, we investigate about some algebraic properties of complete multipartite hypergraphs, like being (almost) complete intersection, Gorenstein, level, $l$-Cohen-Macaulay, $l$-Buchsbaum, unmixed, and satisfying Serre's condition $S_r$, via studying their independent complexes. For a complete $s$-uniform $t$-partite hypergraph $\mathcal{H}$, we show that level, Cohen-Macaulay and $S_r$ properties are equivalent to the condition that all sides of $\mathcal{H}$ have just one vertex, which happens just in the case that the independent complex of $\mathcal{H}$ is a matroid. On the other hand, we show that the independent complex of a complete $s$-uniform $t$-partite hypergraph is a matroid if and only if it is a tight complex. Moreover, we show that when $s>2$, Buchsbaumness is also equivalent to them. Furthermore, we prove that being complete intersection and Gorenstein coincide for complete uniform multipartite hypergraphs. We also characterize all complete uniform multipartite hypergraphs whose edge ideals are $l$-Cohen-Macaulay (resp. $l$-Buchsbaum). Moreover, we determine the shape of those complete multipartite hypergraphs whose edge ideals are almost complete intersection. In Section~\ref{Decomposable}, we prove that vertex decomposability and shellability of the independent complex of a complete $s$-uniform $t$-partite hypergraph $\mathcal{H}$ are equivalent to sequentially $S_r$ and sequentially Cohen-Macaulayness, and all these occur if and only if $\mathcal{H}$ has $t-1$ sides with a single vertex, which is also equivalent to being a chordal hypergraph.

\section{ Preliminaries } \label{Preliminaries}

\noindent In this section, we review some notions and facts around hypergraphs and simplicial complexes. Actually, there are some correspondences between simple hypergraphs and simplicial complexes, which we will mention some of them. Also, we introduce the class of $s$-uniform $t$-partite hypergraphs and mention their basic properties, which we will use in the sequel.

First, recall that a \textbf{simplicial complex} $\Delta$ on the vertex set
$V(\Delta)$ is a collection of subsets of $V(\Delta)$
such that if $F\in \Delta$ and $G\subseteq F$, then $G\in  \Delta$.
An element in $\Delta$ is called a \textbf{face} of $\Delta$, and
$F\in \Delta$ is said to be
a \textbf{facet} if $F$ is maximal with respect to inclusion. We denote the set of facets of $\Delta$ by $\mathcal{F}(\Delta)$. Let $\mathcal{F}(\Delta)=\{F_{1},\ldots,F_{q}\}$. We sometimes write $\Delta=\langle F_{1},\ldots,F_{q}\rangle$. A simplicial complex $\Delta=\langle F_{1},\ldots,F_{q}\rangle$ is \textbf{connected} if for every pair $i,j$, where $1\leq i < j\leq q$, there exists a sequence of facets
$F_{t_1},\ldots,F_{t_r}$ of $\Delta$ such that $F_{t_1}=F_i$, $F_{t_r}=F_j$ and $F_{t_s}\cap F_{t_{s+1}}\neq \emptyset$, for $s=1,\ldots,r-1$.
An \textbf{induced subcomplex} of $\Delta$, over $S\subseteq V(\Delta)$, is defined as $\Delta_S=\{F\in \Delta~:~F\subseteq S\}$.
For every face $G\in \Delta$, we define the \textbf{star} and \textbf{link} of $G$ as
$$\mathrm{st}_{\Delta}G=\{F\in \Delta~:~G\cup F\in \Delta\},$$
$$\mathrm{lk}_{\Delta}G=\{F\in \Delta~:~G\cap F=\emptyset~,~G\cup F\in \Delta\}.$$
The \textbf{dimension} of a face $F$ is $|F|-1$. Let
$d=\textrm{max}\{|F|~:~F\in \Delta \}$. Then the
\textbf{dimension} of $\Delta$, which is denoted by $\textrm{dim}(\Delta)$,
is $d-1$. Let $f_{i}=f_{i}(\Delta)$ denote the number of faces of dimension
$i$. The sequence $f(\Delta)=(f_{0},f_{1},\ldots,f_{d-1})$ is
called the $f$-vector of $\Delta$. By the convention, we set $f_{-1}=1$.
We say that $\Delta$ is \textbf{pure} if all of its facets have the same dimension.
An important class of pure simplicial complexes is the class of matroids. A \textbf{matroid} is a simplicial complex such that if $F$ and $G$ are two faces of it, and $F$ has more elements than $G$, then there exists an element in $F$ which is not in $G$ that when added to $G$ still gives a face. Another class of pure simplicial complexes which contains matroids is the class of tight complexes. A pure simplicial complex $\Delta$ is called a \textbf{tight complex} if there is a labelling of the
vertices such that for every pair of facets $G_1$, $G_2$ and vertices $i\in G_1\setminus G_2$, $j\in G_2\setminus G_1$
with $i<j$, there is a vertex $j'\in G_1\setminus G_2$ such that $(G_2\setminus \{j\})\cup \{j'\}$ is a facet.

A \textbf{transversal} of a simplicial complex $\Delta$ (resp. hypergraph $\mathcal{H}$) is a subset $A$ of its vertex set, with
the property that for every facet $F_{i}$ (resp. edge $e$), $F_{i}\cap A\neq \emptyset$ (resp. $e\cap A\neq \emptyset$).
A \textbf{minimal transversal} of $\Delta$ (resp. $\mathcal{H}$) is a subset $A$ of vertices such that $A$ is a
transversal and no proper subset of $A$ is a transversal of $\Delta$ (resp. $\mathcal{H}$). The minimal
number of vertices of all the minimum transversals of $\Delta$ (resp. $\mathcal{H}$) is called the \textbf{transversal number} of $\Delta$ (resp. $\mathcal{H}$), and denoted by $\tau(\Delta)$ (resp. $\tau(\mathcal{H})$).

The family of minimal transversals of a hypergraph $\mathcal{H}$ constitutes a simple hypergraph on $V(\mathcal{H})$, called the \textbf{transversal hypergraph} of $\mathcal{H}$ and is denoted by $\mathrm{Tr}(\mathcal{H})$. \\
An \textbf{independent} set of $\mathcal{H}$ is a set of vertices which does not contain any edges of $\mathcal{H}$. One can see that a subset of vertices $F$ is maximal independent if and only if $V(\mathcal{H})\setminus F$ is minimal transversal, that is $F\in\mathcal{F}(\mathrm{Ind}(\mathcal{H}))$ if and only if $V(\mathcal{H})\setminus F\in \mathrm{Tr}(\mathcal{H})$. The maximal
number of vertices of all the maximal independent sets of $\mathcal{H}$ is called the \textbf{independent number} of $\mathcal{H}$, and denoted by $i(\mathcal{H})$.
The \textbf{independent complex} of a hypergraph $\mathcal{H}$ is defined as
$$\mathrm{Ind}(\mathcal{H})=\langle F\subseteq V(\mathcal{H})~:~F~\mathrm{is~a~maximal~independent~set~of~}\mathcal{H} \rangle.$$
One can see that ${(\mathrm{Ind}(\mathcal{H}))}_{W}=\mathrm{Ind}({\mathcal{H}}_{W})$, for all $W\subseteq V(\mathcal{H})$. \\
The \textbf{edge complex} of a hypergraph $\mathcal{H}$ is defined as
$$\Delta(\mathcal{H})=\langle e~:~e\in\mathcal{H} \rangle.$$
Similarly, associated to a simplicial complex $\Delta$ is the \textbf{facet hypergraph} which is defined as
$$\mathcal{H}(\Delta)=\{e\subseteq V(\Delta)~:~e\in \mathcal{F}(\Delta)\}.$$

Let $R=K[x_1,\ldots,x_n]$, where $K$ is a field. The \textbf{edge ideal} of a hypergraph $\mathcal{H}$ on $n$ vertices, is the monomial
ideal $$I({\mathcal{H}})=(\textbf{x}^e~:~e\in \mathcal{H}).$$ Note that by $\textbf{x}^e$, we mean $x_{i_1}\cdots x_{i_s}$, where $e=\{x_{i_1},\ldots, x_{i_s}\}$. \\
The \textbf{Stanley-Reisner ideal} of a simplicial complex $\Delta$ is the monomial ideal
$$I_{\Delta}=(\textbf{x}^F~:~F\notin \Delta).$$
One can easily see that $I(\mathcal{H})=I_{{\mathrm{Ind}}(\mathcal{H})}$. The \textbf{Stanley-Reisner ring} of $\Delta$ is defined as $R/I_{\Delta}$.  \\
The \textbf{Alexander dual} of $\Delta$ is the simplicial complex
$$\Delta^{\vee}=\{F^{c} : F\notin \Delta\}.$$
Let $I$ be a squarefree monomial ideal. The \textbf{squarefree
Alexander dual} of $$I=(x_{1,1}\cdots
x_{1,s_1},\ldots,x_{t,1}\cdots x_{t,s_t})$$ is the ideal
$$I^{\vee}=(x_{1,1},\ldots,x_{1,s_1})\cap\cdots\cap(x_{t,1},\ldots,x_{t,s_t}).$$
One has $I_{\Delta^{\vee}}={(I_\Delta)}^{\vee}$. Also, note that $(I(\mathcal{H}))^{\vee}=I(\mathrm{Tr}(\mathcal{H}))$. \\

An $s$-uniform hypergraph $\mathcal{H}$ is said to be $t$-\textbf{partite} (sometimes called $t$-\textbf{colored} $s$-graph), if its vertex set $V(\mathcal{H})$ can be partitioned into sets $V_1, V_2, \ldots, V_t$, called
the \textbf{sides} of $\mathcal{H}$, such that every edge in the edge set $E(\mathcal{H})$ of $\mathcal{H}$ consists of a choice of precisely one vertex from each side. When $s=t$, it is also said that every edge is \textbf{colorful}. An $s$-uniform $t$-partite hypergraph consisting all possible edges in this way, is called the \textbf{complete} $s$-uniform $t$-partite hypergraph. A $t$-partite hypergraph is said to be $m$-\textbf{balanced} if $|V_i|=m$ for every $i=1,\ldots,t$. Throughout the paper, $\mathcal{H}$ is the complete $s$-uniform $t$-partite hypergraph with sides $V_1, V_2, \ldots, V_t$, unless something else is mentioned.

The next proposition follows by definitions and properties mentioned above:

\begin{prop}\label{basic properties}
Let $\mathcal{H}$ be a complete $s$-uniform $t$-partite hypergraph with $2\leq s\leq t$, and $|V_1|\leq |V_2|\leq \cdots \leq |V_t|$. Then we have:\\\\
{\em{(a)}} $\mathrm{Ind}(\mathcal{H})=\big{\langle} \bigcup_{j=1}^{s-1} V_{i_j}~:~1\leq i_1< \cdots <i_{s-1}\leq t\big{\rangle}$. \\\\
{\em{(b)}} $\mathrm{Tr}(\mathcal{H})=\big{\{} \bigcup_{j=1}^{t-s+1}V_{i_j}~:~1\leq i_1< \cdots< i_{t-s+1}\leq t\big{\}}$. \\\\
{\em{(c)}} $i(\mathcal{H})=\sum_{i=t-s+2}^{t}|V_i|$, and $\tau(\mathcal{H})=\sum_{i=1}^{t-s+1}|V_i|$.\\\\
{\em{(d)}} $\mathrm{dim}(R/I(\mathcal{H}))=\mathrm{dim}(\mathrm{Ind}(\mathcal{H}))+1=\sum_{i=t-s+2}^{t}|V_i|$, and $\mathrm{ht}(I(\mathcal{H}))=\sum_{i=1}^{t-s+1}|V_i|$. \\\\
{\em{(e)}} If $\mathcal{H}$ is $m$-balanced for some $m\geq 1$, then $\mathrm{ht}(I(\mathcal{H}))=m(t-s+1)$ and $\mathrm{dim}(R/I(\mathcal{H}))=m(s-1)$.
\end{prop}

\section{Algebraic properties (Gorenstein, Cohen-Macaulay, $S_r$, Buchsbaum, and more) }\label{Algebraic}

\noindent In this section, we study some of the algebraic properties of the edge ideal of a complete $s$-uniform $t$-partite hypergraph. We characterize complete intersection, almost complete intersection, Gorenstein, level, $l$-Cohen-Macaulay, $l$-Buchsbaum, $S_r$, and unmixed complete uniform multipartite hypergraphs. Moreover, we show that level, Cohen-Macaulay and $S_r$ properties are equivalent. On the other hand, we prove that complete intersection and Gorenstein properties coincide. The following theorem characterizes all unmixed edge ideals of complete multipartite hypergraphs. First, note that the edge ideal $I(\mathcal{H})$ is unmixed if and only if $\mathcal{H}$ is an unmixed hypergraph.

\begin{thm}\label{Unmixed}
Let $\mathcal{H}$ be a complete $s$-uniform $t$-partite hypergraph, with $2\leq s\leq t$. Then $I(\mathcal{H})$ is unmixed if and only if
$\mathcal{H}$ is $m$-balanced for some $m\geq 1$.
\end{thm}

\begin{proof}
We have $I(\mathcal{H})$ is unmixed if and only if all the minimal transversals of $\mathcal{H}$ have the same cardinality, that is all elements of
$\mathrm{Tr}(\mathcal{H})$ have the same cardinality. On the other hand, the cardinality of the union of every $t-s+1$ sides of $\mathcal{H}$ is the same if and only if all the sides have the same number of elements. Thus, by Proposition~\ref{basic properties}, $I(\mathcal{H})$ is unmixed if and only if all the sides of $\mathcal{H}$ have the same size, say $m\geq 1$, that is $\mathcal{H}$ is $m$-balanced.
\end{proof}

Let $\Delta$ be a simplicial complex of dimension $d-1$. According to the famous Reisner's criterion, $R/I_\Delta$ is Cohen-Macaulay if and only if $\widetilde{H}_i({\mathrm{lk}}_\Delta F;K)=0$, for all $F\in \Delta$ and all $i<\mathrm{dim}({\mathrm{lk}}_{\Delta}F)$. Also, Terai has formulated the analogue of Reisner's criterion
for Cohen-Macaulay simplicial complexes in the case of $S_r$ simplicial complexes, where $r\geq 2$ (see \cite[page~4, after Theorem~1.7]{T}). Precisely, $R/I_\Delta$ is
$S_r$ if and only if for all $-1\leq i\leq r - 2$ and all
$F\in \Delta$ with $|F|\leq d-i-2$, we have $\widetilde{H}_i({\mathrm{lk}}_\Delta F;K)=0$. \\
Moreover, recall that $R/I_\Delta$ is level if and only if it is Cohen-Macaulay and $\beta_p(R/I_\Delta)=\beta_{p,j}(R/I_\Delta)$, for some $j$, where $p=\mathrm{pd}(R/I_\Delta)$. Now we have the following characterization of complete $s$-uniform $t$-partite hypergraphs:

\begin{thm}\label{CM}
Let $\mathcal{H}$ be a complete $s$-uniform $t$-partite hypergraph with $2\leq s\leq t$, and let $r\geq 2$. Then the following conditions are equivalent:\\
{\em{(a)}} $R/I(\mathcal{H})$ is level. \\
{\em{(b)}} $R/I(\mathcal{H})$ is Cohen-Macaulay.\\
{\em{(c)}} $R/I(\mathcal{H})$ satisfies Serre's condition $S_r$.\\
{\em{(d)}} $\mathcal{H}$ is $1$-balanced, i.e. $\mathcal{H}$ is the complete $s$-uniform hypergraph over $t$ vertices.\\
{\em{(e)}} $\mathrm{Ind}(\mathcal{H})$ is a matroid.\\
{\em{(f)}} $\mathrm{Ind}(\mathcal{H})$ is a tight complex.
\end{thm}

\begin{proof}
(a) $\Rightarrow$ (b) $\Rightarrow$ (c), and (e) $\Rightarrow$ (f) are clear.\\\\
(c) $\Rightarrow$ (d) Suppose that $R/I(\mathcal{H})$ satisfies Serre's condition $S_r$. Then $I(\mathcal{H})$ is unmixed, by \cite[Remark~2.4]{H}. So, by Theorem~\ref{Unmixed}, $\mathcal{H}$ is $m$-balanced for some $m\geq 1$. Suppose on the contrary that $m\geq 2$. Let $F:=\bigcup_{i=1}^{s-2}V_i$ and $\Delta:=\mathrm{Ind}(\mathcal{H})$. By Proposition~\ref{basic properties}, $F\in \Delta$, and also $\mathrm{dim}(\Delta)=m(s-1)-1$. Since $m\geq 2$, we have $|F|=m(s-2)\leq (s-1)m-2$, and hence by the formulation of Terai, $\widetilde{H}_0({\mathrm{lk}}_\Delta F;K)=0$. So, ${\mathrm{lk}}_\Delta F=\langle V_{s-1},V_s,\ldots ,V_t \rangle$ is connected, by \cite[Chapter~0, Proposition~3.3]{S}, which is a contradiction. \\\\
(d) $\Rightarrow$ (e) If $\mathcal{H}$ is $1$-balanced, then $\mathrm{Ind}(\mathcal{H})$ is the collection of all subsets of cardinality less than or equal to $s$ of $t$ vertices. So, by definition, $\mathrm{Ind}(\mathcal{H})$ is a matroid.  \\\\
(e) $\Rightarrow$ (a) follows from \cite[Chapter~III, Theorem~3.4]{S}.\\\\
(f) $\Rightarrow$ (b) If $\mathrm{Ind}(\mathcal{H})$ is a tight complex, then $R/{I(\mathcal{H})}^{(2)}$ is Cohen-Macaulay, by \cite[Theorem 2.5]{MT} (see also \cite{MT1}), where ${I(\mathcal{H})}^{(2)}$ is the second symbolic power of $I(\mathcal{H})$. Thus, $R/I(\mathcal{H})$ is also Cohen-Macaulay, by \cite[Theorem 2.1]{MT}.
\end{proof}

When $s=2$, we have a complete $t$-partite graph. So, we deduce the following result of \cite{SRS}:

\begin{cor}\label{graph}
Let $G$ be a complete $t$-partite graph. Then $R/I(G)$ is Cohen-Macaulay if and only if $G$ is the complete graph over $t$ vertices.
\end{cor}

There is also a formulation for Buchsbaumness  with respect to the reduced homology groups of a simplicial complex $\Delta$ (see \cite[Theorem~2]{M}). According to it,
$R/I_\Delta$ is Buchsbaum if and only if $\Delta$ is pure and $\widetilde{H}_i({\mathrm{lk}}_\Delta F;K)=0$, for all $F\in \Delta \setminus \{\emptyset\}$ and all $i<\mathrm{dim}({\mathrm{lk}}_{\Delta}F)$. So, if $R/I_\Delta$ is Buchsbaum, then it is Cohen-Macaulay.
As a generalization of Buchsbaumness, we recall $l$-Buchsbaum property ($l$-Cohen-Macaulay property is defined similarly). Precisely, $R/I_\Delta$ is called $l$-Buchsbaum (resp. $l$-Cohen-Macaulay) if for all $W\subset V=V(\Delta)$ with $|W|<l$, one has $R/I_{\Delta_{V\setminus W}}$ is Buchsbaum (resp. Cohen-Macaulay) and $\mathrm{dim}(\Delta)=\mathrm{dim}(\Delta_{V\setminus W})$ (see \cite[page~370]{M1}). Obviously, $1$-Buchsbaum (resp. $1$-Cohen-Macaulay) and Buchsbaum (resp. Cohen-Macaulay) properties coincide.

The following theorem shows that whenever $s\geq 3$, Buchsbaum property is also equivalent to those mentioned in Theorem \ref{CM}. More generally, we have:

\begin{thm}\label{Bbm-s>2}
Let $\mathcal{H}$ be a complete $s$-uniform $t$-partite hypergraph, with $3\leq s\leq t$, and let $l\geq 1$. Then $R/I(\mathcal{H})$ is $l$-Buchsbaum if and only if $\mathcal{H}$ is $1$-balanced and $l\leq t-s+2$. In particular, Buchsbaum and $2$-Buchsbaum properties for $R/I(\mathcal{H})$ coincide, and they are equivalent to the conditions in Theorem \ref{CM}.
\end{thm}

\begin{proof}
Let $\Delta:=\mathrm{Ind}(\mathcal{H})$. If $R/I(\mathcal{H})$ is $l$-Buchsbaum, then it is Buchsbaum, and hence unmixed. So, by Theorem~\ref{Unmixed}, $\mathcal{H}$ is $m$-balanced for some $m\geq 1$. Suppose that $m\geq 2$, and let $F:=\bigcup_{i=1}^{s-2}V_i$. Since $s\geq 3$, we have $\emptyset \neq F\in \Delta$, by Proposition~\ref{basic properties}. Using a similar argument as in the proof of Theorem~\ref{CM}, we have $\widetilde{H}_0({\mathrm{lk}}_\Delta F;K)=0$, since $m\geq 2$, ${\mathrm{lk}}_\Delta F=\langle V_{s-1},V_s,\ldots ,V_t \rangle$, and so $\mathrm{dim}({\mathrm{lk}}_\Delta F)=m-1$. Hence, ${\mathrm{lk}}_\Delta F$ is connected, a contradiction. So, $\mathcal{H}$ is $1$-balanced. Therefore, $\Delta$ is the simplicial complex over $t$ vertices, whose facets are all subsets of cardinality $s-1$ of the vertex set, and $\mathrm{dim}(\Delta)=s-2$.
On the other hand, since $\mathcal{H}$ is $l$-Buchsbaum,
we have $\mathrm{dim}(\Delta)=\mathrm{dim}(\Delta_{V\setminus W})$, for all $W\subset V=V(\Delta)$ with $|W|<l$. Thus, we have $s-1\leq t-(l-1)$, and so $l\leq t-s+2$. Conversely, suppose that $\mathcal{H}$ is $1$-balanced and $l\leq t-s+2$. Assume that $W\subset V$ such that $|W|<l$. Hence, $s\leq t-|W|+1$. If $s=t-|W|+1$, then ${\mathcal{H}}_{V\setminus W}$ does not have any edges, so that $\Delta_{V\setminus W}=\mathrm{Ind}({\mathcal{H}}_{V\setminus W})$ is a simplex on $V\setminus W$ of dimension $t-|W|-1=s-2$. Thus, $R/I({\mathcal{H}}_{V\setminus W})$ is Cohen-Macaulay. If $s\leq t-|W|$, then ${\mathcal{H}}_{V\setminus W}$ is a $1$-balanced complete $s$-uniform $(t-|W|)$-partite hypergraph. Thus, by Theorem~\ref{CM}, $R/I({\mathcal{H}}_{V\setminus W})$ is Cohen-Macaulay. Moreover, by Proposition~\ref{basic properties}, $\mathrm{dim}(\Delta_{V\setminus W})=\mathrm{dim}(\Delta)=s-2$. Therefore, in both cases, $R/I(\mathcal{H})$ is $l$-Cohen-Macaulay, and hence $l$-Buchsbaum.
\end{proof}

The next theorem completes the characterization of the edge ideal of complete $s$-uniform $t$-partite hypergraphs with respect to Buchsbaum property.
We denote an $m$-balanced complete $t$-partite graph, where $m\geq 1$ and $t\geq 2$, by $K_{t\ast m}$.

\begin{thm}\label{Bbm-s=2}
Let $G$ be a complete $t$-partite graph with $t\geq 2$. Then the following conditions are equivalent:\\
{\em{(a)}} $R/I(G)$ is Buchsbaum.\\
{\em{(b)}} $I(G)$ is unmixed.\\
{\em{(c)}} $G$ is just $K_{t\ast m}$, for some $m\geq 1$.
\end{thm}

\begin{proof}
(a) $\Rightarrow$ (b) is clear. \\
(b) $\Rightarrow$ (c) follows from Theorem~\ref{Unmixed}. \\
(c) $\Rightarrow$ (a) Suppose that $G$ is of the form $K_{t\ast m}$, for some $m\geq 1$, and $\Delta:=\mathrm{Ind}(G)$. So, $\Delta=\langle V_1,\ldots,V_t\rangle$. Let $F$ be a nonempty face of $\Delta$. So, it is contained in some $V_i$. Thus, ${\mathrm{lk}}_{\Delta}F$ is either $\{\emptyset \}$ or a simplex, so that $\widetilde{H}_i({\mathrm{lk}}_\Delta F;K)=0$, for all $i$. Therefore, $R/I(G)$ is Buchsbaum.
\end{proof}

\begin{rem}\label{l-Bbm,s=2}
{\em Note that when $G$ is a complete $t$-partite graph, $R/I(G)$ is not $l$-Buchsbaum, for $l\geq 2$, unless $G$ is complete. More precisely, let $G$ be a complete $t$-partite graph on $V$. Also, let $V_1,\ldots,V_t$ be the sides of $G$. If $G$ is non-complete, then $|V_i|\geq 2$, for all $i=1,\ldots,t$. Now, if $R/I(G)$ is $l$-Buchsbaum, for some $l\geq 2$, then it is Buchsbaum and hence it is just $K_{t*m}$, for some $m\geq 2$, by Theorem~\ref{Bbm-s=2}. Let $v\in  V_1$. Then $R/I(G_{V\setminus \{v\}})$ is Buchsbaum, by the definition. But this is a contradiction, again by Theorem~\ref{Bbm-s=2}, because $V_1$ has one vertex less than the other sides of $G_{V\setminus \{v\}}$. }
\end{rem}

The following theorem could be seen as a generalization of the equivalency of parts (b) and (d) of Theorem \ref{CM}:

\begin{thm}\label{l-CM}
Let $\mathcal{H}$ be a complete $s$-uniform $t$-partite hypergraph with $2\leq s\leq t$, and $l\geq 1$. Then $R/I(\mathcal{H})$ is $l$-Cohen-Macaulay if and only if $\mathcal{H}$ is $1$-balanced and $l\leq t-s+2$. In particular, Cohen-Macaulay and $2$-Cohen-Macaulay properties are equivalent for $R/I(\mathcal{H})$.
\end{thm}

\begin{proof}
If $R/I(\mathcal{H})$ is $l$-Cohen-Macaulay, then it is also Cohen-Macaulay, and hence unmixed. So, by Theorem~\ref{CM}, $\mathcal{H}$ is $1$-balanced.
Thus, similar to the proof of Theorem~\ref{Bbm-s>2}, $\Delta:=\mathrm{Ind}(\mathcal{H})$ is the simplicial complex over $t$ vertices, whose facets are all subsets of cardinality $s-1$ of the vertex set, and $\mathrm{dim}(\Delta)=s-2$.
Since $\mathcal{H}$ is $l$-Buchsbaum,
we have $\mathrm{dim}(\Delta)=\mathrm{dim}(\Delta_{V\setminus W})$, for all $W\subset V=V(\Delta)$ with $|W|<l$. Thus, we have $s-1\leq t-(l-1)$, and so $l\leq t-s+2$. The converse is actually proved in Theorem~\ref{Bbm-s>2}.
\end{proof}

Recall that for a simplicial complex $\Delta$, $R/I_\Delta$ is complete intersection if $\mu(I_\Delta)=\mathrm{ht}(I_\Delta)$, where by $\mu(I_\Delta)$, we mean the number of minimal generators of $I_\Delta$. Now, we want to characterize all complete uniform multipartite hypergraphs, whose edge ideals are complete intersection and  Gorenstein. For this purpose, the following characterization of Gorenstein simplicial complexes is needed:

\begin{thm}\label{Stanley}
\cite[Chapter II, Theorem 5.1]{S} Fix a field $k$ (or $\mathbb{Z}$). Let $\Delta$ be a simplicial complex and
$\Lambda:=\mathrm{core}(\Delta)$. Then the following are equivalent:\\
\indent {\em {(i)}} $\Delta$ is Gorenstein.\\
\indent {\em {(ii)}} either {\em (1)} $\Delta=\emptyset$, o , or o o , or {\em {(2)}} $\Delta$ is Cohen-Macaulay over $k$ of dimension
\indent $d-1\geq 1$, and the link of every $(d-3)$-face is either a circle or o-o or o-o-o , and
\indent $\widetilde{\chi}(\Lambda)={(-1)}^{\mathrm{dim(\Lambda)}}$ {\em (}the last condition
is superfluous over $\mathbb{Z}$ or if $\mathrm{char}(k)=2${\em )}.\\
Here, $\mathrm{core}(\Delta)=\Delta_{\mathrm{core}(V)}$, in which $\mathrm{core}(V)=\{v\in V~:~\mathrm{st}_{\Delta}\{v\}\neq \Delta\}$, and $\widetilde{\chi}(\Delta)$ is the reduced Euler characteristic of $\Delta$ which is given by $\widetilde{\chi}(\Delta)=-1+\sum_{i=0}^{d-1}(-1)^if_{i}$.
\end{thm}

\begin{thm}\label{Gor}
Let $\mathcal{H}$ be a complete $s$-uniform $t$-partite hypergraph with $2\leq s\leq t$. Then the following conditions are equivalent:\\
{\em{(a)}} $R/I(\mathcal{H})$ is complete intersection. \\
{\em{(b)}} $R/I(\mathcal{H})$ is Gorenstein.\\
{\em{(c)}} $\mathcal{H}$ is $1$-balanced and $s=t$, i.e. $\mathcal{H}$ has only one edge.
\end{thm}

\begin{proof}
(a) $\Rightarrow$ (b) and (c) $\Rightarrow$ (a) are clear. \\
(b) $\Rightarrow$ (c) Suppose that $R/I(\mathcal{H})$ is Gorenstein. So, it is also Cohen-Macaulay and hence $1$-balanced with vertex set $\{v_1,\ldots,v_t\}$, by Theorem~\ref{CM}. Let $\Delta:=\mathrm{Ind}(\mathcal{H})$. If $s=2$, then $\mathrm{dim}(\Delta)=0$, by Proposition~\ref{basic properties}. So, Theorem~\ref{Stanley}, implies that $\Delta$ is
of the form o o, which is a contradiction. So, suppose that $3\leq s$. Then $\mathrm{dim}(\Delta)\geq 1$, again by Proposition~\ref{basic properties}. Let
$F:=\{v_1,\ldots,v_{s-3}\}$. So, ${\mathrm{lk}}_{\Delta}F=\langle \{v_i,v_j\}~:~s-2\leq i< j\leq t\rangle$, that is the complete graph on $t-s+3$ vertices.
Thus, using Theorem~\ref{Stanley}, this complete graph must be a $3$-cycle, so that $t=s$, as desired.
\end{proof}

Let $\Delta$ be a simplicial complex, then $R/I_\Delta$ is called \textbf{almost complete intersection} if $\mu(I_\Delta)=\mathrm{ht}(I_\Delta)+1$. The next theorem determines all the complete uniform multipartite hypergraphs with almost complete intersection edge ideals.

\begin{thm}\label{almost c.i.}
Let $\mathcal{H}$ be a complete $s$-uniform $t$-partite hypergraph with $2\leq s\leq t$. Then $R/I(\mathcal{H})$ is almost complete intersection if and only if one of the following holds:\\
{\em{(a)}} $\mathcal{H}$ is $C_3$, i.e. the $3$-cycle, or \\
{\em{(b)}} $\mathcal{H}$ is the hypergraph over $t+1$ vertices $V=\{v_1,\ldots,v_{t+1}\}$, and with edges $\{v_1,\ldots ,v_{t-1},v_t\}$ and $\{v_1,\ldots,v_{t-1},v_{t+1}\}$.
\end{thm}

\begin{proof}
If $\mathcal{H}$ is $C_3$, then $\mu(I(\mathcal{H}))=3$ and $\mathrm{ht}(I(\mathcal{H}))=\tau(\mathcal{H})=2$. If $\mathcal{H}$ is the hypergraph over $V$ with edges $\{v_1,\ldots ,v_{t-1},v_t\}$ and $\{v_1,\ldots,v_{t-1},v_{t+1}\}$, then $\mu(I(\mathcal{H}))=2$ and $\mathrm{ht}(I(\mathcal{H}))=\tau(\mathcal{H})=1$. So, in both cases, we have that $R/I(\mathcal{H})$ is almost complete intersection. Conversely, suppose that $R/I(\mathcal{H})$ is almost complete intersection and $|V_1|\leq \cdots \leq |V_t|$. Now, we consider the following cases: \\
(1) Suppose that $s=t$. Then $\tau(\mathcal{H})=|V_1|$, by Proposition~\ref{basic properties}, so that $\mu(I(\mathcal{H}))=|V_1|+1$. If $|V_1|\geq 2$, then $|V_2|\geq 2$, by our assumption. Thus, each vertex of $V_1$ is contained in at least two edges. So, the number of edges of $\mathcal{H}$ and hence $\mu(I(\mathcal{H}))$ is more than or equal to $2|V_1|$, which is a contradiction, since $\mu(I(\mathcal{H}))=|V_1|+1$ and we assumed that $|V_1|\geq 2$. So, we have $|V_1|=1$. Therefore, $\mu(I(\mathcal{H}))=2$, that is $\mathcal{H}$ has exactly two edges. Thus, we have
$|V_1|=\cdots=|V_{t-1}|=1$ and $|V_t|=2$. So, $V_i=\{v_i\}$, for all $i=1,\ldots,t-1$, and $V_t=\{v_t,v_{t+1}\}$. Hence, in this case, $\mathcal{H}$ is of the form (b).  \\
(2) Suppose that $s<t$. Then $\tau(\mathcal{H})=\sum_{i=1}^{t-s+1}|V_i|$, by Proposition~\ref{basic properties}, so that $\mu(I(\mathcal{H}))=\sum_{i=1}^{t-s+1}|V_i|+1$. Suppose that $|V_t|\geq 2$. Let $W_i:=V_i\cup (\bigcup_{j=t-s+2}^tV_j)$, for all $i=1,\ldots,t-s+1$. Then, the induced subhypergraph ${\mathcal{H}}_{W_i}$ has more than or equal to $2|V_i|$ edges, for all $i=1,\ldots,t-s+1$. Also, note that ${\mathcal{H}}_{W_i}$ and ${\mathcal{H}}_{W_j}$ are pairwise disjoint, for all $i\neq j$. Thus, the number of edges of $\mathcal{H}$ is more than or equal to $2\sum_{i=1}^{t-s+1}|V_i|$, so that $\mu(I(\mathcal{H}))\geq 2\sum_{i=1}^{t-s+1}|V_i|$, which is a contradiction, because $\mu(I(\mathcal{H}))=\sum_{i=1}^{t-s+1}|V_i|+1$. Therefore, $|V_t|=1$, which implies that $|V_i|=1$, for all $i=1,\ldots,t$, by our assumption. So, $V=\{v_1,\ldots,v_t\}$ and $\mathcal{H}$ is the set of all subsets of cardinality of $s$ of $V$. So, the number of edges of $\mathcal{H}$ is ${t\choose s}$. On the other hand, $\mu(I(\mathcal{H}))=t-s+2$. Hence, ${t\choose s}=t-s+2$. But, it occurs if and only if $s=2$ and $t=3$. Therefore, in this case, $\mathcal{H}$ is just $C_3$.
\end{proof}

We end this section by some properties of the Alexander dual of the edge ideal of a complete $s$-uniform $t$-partite hypergraph. First, recall that a homogeneous ideal $I$ whose generators all have degree $d$ is said to have a \textbf{$d$-linear resolution} (or simply linear resolution) if
for all $i\geq0$, $\beta_{i,j}(I)=0$ for all $j\neq{i+d}$. By \cite[Theorem~3.2]{MKM}, one has $I(\mathcal{H})$ is weakly polymatroidal, and hence has a linear resolution. So, we have

\begin{prop}\label{Alexander dual}
Let $\mathcal{H}$ be a complete $s$-uniform $t$-partite hypergraph with $2\leq s\leq t$. Then $R/(I(\mathcal{H}))^{\vee}$ is Cohen-Macaulay. Moreover,
$R/(I(\mathcal{H}))^{\vee}$ is complete intersection if and only if $t=s$.
\end{prop}

\begin{proof}
The first statement follows easily by using \cite[Theorem~3]{ER}. As we mentioned before, one has $(I(\mathcal{H}))^{\vee}=I(\mathrm{Tr}(\mathcal{H}))$. On the other hand, for a monomial ideal $J$, we know that $R/J$ is complete intersection if and only if $J$ is generated by monomials which are pairwise coprime. Thus, combining these facts together with Proposition~\ref{basic properties}, we get the second statement.
\end{proof}

\section{ Vertex decomposable, Shellable and Sequentially Cohen-Macaulay Complete Multipartite Hypergraphs } \label{Decomposable}

\noindent In this section, we investigate about those complete uniform multipartite hypergraphs, whose independent complexes are decomposable and shellable. In fact, we show that for a complete $s$-uniform $t$-partite hypergraph with $2\leq s\leq t$, $(s-2)$-decomposability, shellability, sequentially Cohen-Macaulayness, and sequentially $S_r$ property are equivalent. Moreover, we obtain a characterization with respect to hypergraph's terms for them. On the other hand, we determine all chordal complete multipartite hypergraphs, in sense of \cite{W}. The following theorem is the main result of this section:

\begin{thm}\label{seq}
Let $\mathcal{H}$ be a complete $s$-uniform $t$-partite hypergraph with $2\leq s\leq t$, and let $r\geq 2$. Then the following conditions are equivalent:\\
{\em{(a)}} $\mathrm{Ind}(\mathcal{H})$ is vertex decomposable.\\
{\em{(b)}} $\mathrm{Ind}(\mathcal{H})$ is shellable.\\
{\em{(c)}} $R/I(\mathcal{H})$ is sequentially Cohen-Macaulay. \\
{\em{(d)}} $R/I(\mathcal{H})$ is sequentially $S_r$.\\
{\em{(e)}} $\mathcal{H}$ has $(t-1)$ sides consisting of exactly one vertex. \\
{\em{(f)}} $\mathcal{H}$ is a chordal hypergraph.
\end{thm}

To prove this theorem, we need some notions and facts that we will mention in the sequel.

Let $\Delta$ be a $(d-1)$-dimensional simplicial complex. For each $i=0,\ldots,d-1$, the \textbf{pure $i$-th skeleton} of $\Delta$ is defined to be the pure
subcomplex $\Delta^{[i]}$ of $\Delta$ whose facets are those faces $F$ of $\Delta$ with $|F|=i+1$.

A graded $R$-module $M$ is called \textbf{sequentially}
Cohen-Macaulay (resp. $S_r$) (over $K$) if there exists a finite filtration of graded $R$-modules
$0=M_{0}\subset M_{1}\subset\cdots \subset M_{r}=M$
such that each $M_{i}/M_{i-1}$ is Cohen-Macaulay (resp. $S_r$), and the Krull dimensions of the quotients
are increasing, i.e.
$$ \mathrm {dim} (M_{1}/M_{0}) < \mathrm{dim}(M_{2}/M_{1}) < \cdots < \mathrm{dim}(M_{r}/M_{r-1}).$$
In \cite{HTYZ}, the authors gave a criterion for being sequentially $S_2$ of Stanley-Reisner rings. They showed that

\begin{thm}\label{seq S_2}
\cite[Theorem~2.9]{HTYZ} Let $\Delta$ be a simplicial complex with vertex set $V$. Then $R/I_\Delta$ is sequentially $S_2$ if and only if $\Delta^{[i]}$ is connected for all $i\geq 1$, and $R/I_{\mathrm{lk}_\Delta\{x\}}$ is sequentially $S_2$ for all $x\in V$.
\end{thm}

Let $\Delta$ be a $(d-1)$-dimensional simplicial complex on vertex set $V$. Then it is said to be (nonpure) \textbf{shellable} if its facets can be ordered $F_1,F_2,\ldots,F_m$ such that for all $i$ and all $j<i$, there exists $v_l\in F_i\setminus F_j$ and $k<i$ such that $F_i \setminus F_k=\{v_l\}$. If $\Delta$
is shellable, then $R/I_\Delta$ is sequentially Cohen-Macaulay (see for example \cite[Corollary~8.2.19]{HH}). \\
A simplicial complex $\Delta$ is recursively defined to be \textbf{vertex decomposable} if it is
either a simplex or else has some vertex $v$ so that \\
\indent (1) both $\Delta \setminus v$ and $\mathrm{lk}_\Delta \{v\}$ are vertex decomposable, and \\
\indent (2) no face of $\mathrm{lk}_\Delta \{v\}$ is a facet of $\Delta \setminus v$.\\
Here, by $\Delta\setminus v$, we mean the simplicial complex which is obtained from $\Delta$ by removing all faces that contain $v$ from $\Delta$. A vertex $v$ which satisfies Condition (2) is called a \textbf{shedding} vertex. Moreover, complexes $\{\}$ and $\{\emptyset\}$ are considered to be vertex decomposable. One has that every vertex decomposable simplicial complex is shellable. (see \cite[Section~2]{W1}).

Now, let $\mathcal{H}$ be a hypergraph on the vertex set $V$, and let $v\in V$. The \textbf{deletion} $\mathcal{H}\setminus v$ is the hypergraph on the vertex set $V\setminus \{v\}$ with edges $\{e~:~e~\mathrm{is~an~edge~of}~\mathcal{H}~\mathrm{with}~v\notin e\}$. The \textbf{contraction} $\mathcal{H}/v$ is the hypergraph on the vertex set $V\setminus \{v\}$ with edges the minimal sets of $\{e\setminus \{v\}~:~e~\mathrm{is~an~edge~of}~\mathcal{H}\}$, that is $\mathcal{H}/v$ removes $v$ from each edge containing it, and then removes any redundant edges. A hypergraph $\mathcal{H'}$ obtained from $\mathcal{H}$ by repeated deletion and/or contraction is called a \textbf{minor} of $\mathcal{H}$.

\begin{lem}\label{contraction-deletion}
\cite[Lemma~2.2]{W} Let $\mathcal{H}$ be a hypergraph and $v\in V(\mathcal{H})$. Then \\
{\em{(a)}} $\mathrm{Ind}(\mathcal{H}/v)=\mathrm{lk}_{\mathrm{Ind}(\mathcal{H})}\{v\}$, \\
{\em{(b)}} $\mathrm{Ind}(\mathcal{H}\setminus v)=\mathrm{Ind}(\mathcal{H})\setminus v$.
\end{lem}

A vertex $v$ of $\mathcal{H}$ is called \textbf{simplicial} if for every two edges $e_1$ and $e_2$ of $\mathcal{H}$ that contain $v$, there is a third edge $e_3$ such that $e_3\subseteq (e_1\cup e_2)\setminus \{v\}$. A hypergraph $\mathcal{H}$ is \textbf{chordal} if every minor of $\mathcal{H}$ has a simplicial vertex. Note that this definition coincide with the original one for simple graphs. (See \cite[Sections~2~and~4]{W} for latter definitions. Note that here we use the words \textit{hypergraph} and \textit{edge} instead of what is called \textit{clutter} and \textit{circuit}, respectively, in \cite{W}).

Recall that the \textbf{join} of two simplicial complexes $\Delta_1$ and $\Delta_2$ on disjoint vertex sets $V_1$ and $V_2$ is a simplicial complex, denoted by $\Delta_1\ast \Delta_2$, on vertex set $V_1\cup V_2$ with faces $\{F_1\cup F_2~:~F_1\in \Delta_1~,~F_2\in \Delta_2\}$. \\

{\em Proof of Theorem \ref{seq}.}
(a) $\Rightarrow$ (b) $\Rightarrow$ (c) $\Rightarrow$ (d) are well-known, as we mentioned above. \\

In the rest of the proof, assume that $|V_1|\leq \cdots \leq |V_t|$. \\

(d) $\Rightarrow$ (e) Suppose that $R/I(\mathcal{H})$ is sequentially $S_r$, so that it is sequentially $S_2$. Let $\Delta:=\mathrm{Ind}(\mathcal{H})$. Also, assume that $V_i=\{v_{i1},\ldots,v_{in_i}\}$, for all $i=1,\ldots,t$. Now, suppose on the contrary that $|V_{t-1}|\geq 2$, so that $|V_t|\geq 2$. For all $i=1,\ldots,s-2$ and all $j=1,\ldots,n_i$, we set $\Gamma_{i,j}:=\langle V_i\setminus \{v_{i1},\ldots,v_{ij}\}\rangle$, which is a simplex, and ${\mathcal{H}}_i$ to be the complete $(s-i)$-uniform $(t-i)$-partite hypergraph over the set of vertices $\bigcup_{l=i+1}^t V_l$. Now, set $\Delta_{i,j}:=\Gamma_{i,j}\ast \mathrm{Ind}(\mathcal{H}_i)$, for all $i=1,\ldots,s-2$ and all $j=1,\ldots,n_i$. One can see that $\Delta_{i,1}=\mathrm{lk}_{\Delta_{i-1,n_{i-1}}}\{v_{i1}\}$, and $\Delta_{i,j}=\mathrm{lk}_{\Delta_{i,j-1}}\{v_{ij}\}$, for all $j=2,\ldots,n_i$. Thus, for example, $\Delta_{1,1}=\Gamma_{1,1}\ast \mathrm{Ind}(\mathcal{H}_1)=\mathrm{lk}_{\Delta}\{v_{11}\}$, where $R/I_{\Delta_{1,1}}$ is sequentially $S_2$, by Theorem~\ref{seq S_2}. So, again by using Theorem~\ref{seq S_2}, we have that the Stanley-Reisner ring of $\Delta_{1,2}=\mathrm{lk}_{\Delta_{1,1}}\{v_{12}\}$ is sequentially $S_2$. Then, using repeatedly Theorem~\ref{seq S_2}, we have that the Stanley-Reisner ring of $\Delta_{i,j}$ is sequentially $S_2$, for all $i=1,\ldots,s-2$ and all $j=1,\ldots,n_i$. In particular, the Stanley-Reisner ring of $\Delta_{s-2,n_{s-2}}=\Gamma_{s-2,n_{s-2}}\ast \mathrm{Ind}(\mathcal{H}_{s-2})$ is sequentially $S_2$. Thus, $\Delta_{s-2,n_{s-2}}^{[1]}$ is a connected graph, by Theorem~\ref{seq S_2}. But, $\Gamma_{s-2,n_{s-2}}=\{\emptyset\}$, and hence $\Delta_{s-2,n_{s-2}}=\mathrm{Ind}(\mathcal{H}_{s-2})=\langle V_{s-1},\ldots,V_{t}\rangle$, by Proposition~\ref{basic properties}. Therefore, $\Delta_{s-2,n_{s-2}}^{[1]}$ has at least two connected components over two disjoint sets of vertices $V_{t-1}$ and $V_t$, because $|V_{t}|\geq |V_{t-1}|\geq 2$, a contraction. So, we have $|V_{t-1}|=1$, which implies that $|V_i|=1$, for all $i=1,\ldots,t-1$, by our assumption. \\

(e) $\Rightarrow$ (a) Suppose that $\mathcal{H}$ has $t-1$ sides consisting of exactly one vertex. Let $V:=V(\mathcal{H})$, $n:=|V|\geq 2$ and $\Delta:=\mathrm{Ind}(\mathcal{H})$. By Proposition~\ref{basic properties}, we have $\Delta=\langle \{v_{i_1},\ldots,v_{i_{s-1}}\},V_t\cup \{v_{i_1},\ldots,v_{i_{s-2}}\}~:~1\leq i_1<\cdots <i_{s-1}\leq t-1 \rangle$. Without loss of generality, we assume that $V_{j}:=\{v_j\}$, for all $j=1,\ldots,t-1$, and let $V_{t}:=\{v_t,\ldots,v_n\}$. If $s=t$, then $\mathcal{H}\setminus v_1$ does not have any edges, and hence $\mathrm{Ind}(\mathcal{H}\setminus v_1)$ is a simplex on $V\setminus \{v_1\}$, which is vertex decomposable. If $s<t$, then one can see that $\mathcal{H}\setminus v_1$ is a complete $s$-uniform $(t-1)$-partite hypergraph over $V\setminus \{v_1\}$. Thus, $\mathrm{Ind}(\mathcal{H}\setminus v_1)=\langle \{v_{i_1},\ldots,v_{i_{s-1}}\},V_t\cup \{v_{i_1},\ldots,v_{i_{s-2}}\}~:~1<i_1<\cdots <i_{s-1}\leq t-1 \rangle$. Also, one can see that $\mathcal{H}/ v_1$ is a complete $(s-1)$-uniform $(t-1)$-partite hypergraph over $V\setminus \{v_1\}$. Thus, $\mathrm{Ind}(\mathcal{H}/ v_1)=\langle \{v_{i_1},\ldots,v_{i_{s-2}}\},V_t\cup \{v_{i_1},\ldots,v_{i_{s-3}}\}~:~1<i_1<\cdots <i_{s-2}\leq t-1 \rangle$. So, by Lemma~\ref{contraction-deletion}, no face of $\mathrm{lk}_{\Delta}\{v_1\}$ is a facet of $\Delta\setminus v_1$. So, $v_1$ is a shedding vertex. Now, the result follows by induction on $|V(\mathcal{H})|\geq 2$. \\

(e) $\Rightarrow$ (f) Suppose that $\mathcal{H}$ has $t-1$ sides consisting of exactly one vertex. If $|V_t|=1$, then $\mathcal{H}$ is an $s$-uniform complete hypergraph and hence chordal, by \cite[Example~4.4]{W}. Let $|V_t|>1$. If $v\in V_t$, then we show that $v$ is a simplicial vertex of $\mathcal{H}$. Let $e_1=\{v_{i_1},\ldots,v_{i_{s-1}},v\}$ and $e_2=\{v_{j_1},\ldots,v_{j_{s-1}},v\}$ be two distinct edges of $\mathcal{H}$, where $v_{i_1},\ldots,v_{i_{s-1}}$ and $v_{j_1},\ldots,v_{j_{s-1}}$ are contained in sides of cardinality $1$. Then, there exists $j_t$ such that $1\leq t\leq s-1$ and $v_{j_t}\notin \{v_{i_1},\ldots,v_{i_{s-1}}\}$. Thus, $e_3=\{v_{j_t},v_{i_1},\ldots,v_{i_{s-1}}\}$ is an edge of $\mathcal{H}$, since $v_{j_t}$ is not in the $s-1$ other sides containing the vertices $v_{i_1},\ldots,v_{i_{s-1}}$. So, $e_3\subseteq (e_1\cup e_2)\setminus \{v\}$, and hence by the definition, $v$ is a simplicial vertex of $\mathcal{H}$. Now, we show that all minors of $\mathcal{H}$ have a simplicial vertex. If $v\in V_i$, for some $i=1,\ldots,t-1$, then similar to the proof of (e) $\Rightarrow$ (a), $\mathcal{H}\setminus v$ has no edges or is a complete $s$-uniform $(t-1)$-partite hypergraph over $V\setminus \{v\}$, and $\mathcal{H}/ v$ is a complete $(s-1)$-uniform $(t-1)$-partite hypergraph over $V\setminus \{v\}$. If $v\in V_t$, then one can see that $\mathcal{H}\setminus v$ is a complete $s$-uniform $t$-partite hypergraph over $V\setminus \{v\}$, and $\mathcal{H}/ v$ is an $(s-1)$-uniform complete hypergraph over $V\setminus V_t$. Thus, in all these cases, the obtained hypergraph has a simplicial vertex, as mentioned above for $\mathcal{H}$. Thus, by a similar discussion, all minors of $\mathcal{H}$ have a simplicial vertex, so that $\mathcal{H}$ is chordal. \\

(f) $\Rightarrow$ (b) follows by \cite[Corollary 5.4]{W}.
~~~~~~~~~$~~~\Box$ \\

Let $G$ and $H$ be two graphs. We denote by $G*H$, the join of two graphs $G$ and $H$, that is the graph with vertex set $V(G)\cup V(H)$, and the edge set $E(G)\cup E(H)\cup \{\{v,w\}~:~v\in V(G),~w\in V(H)\}$. Moreover, we denote the complete graph on $n$ vertices, by $K_n$, and the complementary of a graph $G$, by $G^c$.

\begin{cor}\label{seq-graph}
Let $G$ be a complete $t$-partite graph. Then $R/I(G)$ is sequentially Cohen-Macaulay if and only if $G$ is $K_{t-1}*K_m^c$, for some $m\geq 1$.
\end{cor}

The equivalency of parts (a), (b) and (e) of Theorem~\ref{seq} generalizes the results of \cite{SRS} for multipartite graphs. \\

\textbf{Acknowledgments:} We would like to thank Russ Woodroofe who motivated us to look at the vertex decomposability and chordalness in Theorem~\ref{seq}. We would also like to thank the referee for his or her useful and valuable comments.  Moreover, the authors would like to thank to the Institute for Research in Fundamental Sciences (IPM) for financial support. The research of the first author was in part supported by a grant from IPM (No. 93050220).

\providecommand{\byame}{\leavevmode\hbox
to3em{\hrulefill}\thinspace}

\end{document}